\documentclass[12pt]{amsart}
\usepackage{verbatim}
\usepackage[numbers]{natbib}
\usepackage{amsmath}
\usepackage{amssymb}
\usepackage{calc}
\usepackage{amsmath}
\usepackage{amssymb}
\usepackage{amsthm}
\usepackage{relsize}
\usepackage{bm}
\usepackage{hyperref}
\usepackage{tikz-cd}
\usepackage[utf8]{inputenc}
\usepackage[english]{babel}
\usepackage[margin=1in]{geometry}
\usepackage{graphicx}
\usepackage{subfigure}

\newtheorem{theorem}{Theorem}[section]

\newtheorem{proposition}[theorem]{Proposition}
\newtheorem{lemma}[theorem]{Lemma}
\newtheorem{corollary}[theorem]{Corollary}

\theoremstyle{definition}

\newtheorem*{example}{Example}
\newtheorem*{remark}{Remark}

\title{A Counterexample to a Question on Grothendieck Groups of Schemes}
\author{Amal Mattoo}

\begin{document} 
	\begin{abstract}
		If an element of the Grothendieck group of the derived category of a scheme is locally represented by perfect complexes, then can the original element be represented by a perfect complex? We provide a counterexample on a projective variety of dimension $2$, as well as a counterexample on a thickening of a Dedekind domain. 
	\end{abstract}
	\maketitle
	\section{Introduction}
		This paper answers the following question posed by Andrei Okounkov. 
		
		Let $X$ be a scheme, $\bigcup_{i}U_{i}$ an open covering, and $[\alpha]\in K_{0}'(X)$ an element of the Grothendieck group of the bounded derived category of coherent sheaves on $X$. Suppose that $[\alpha|_{U_{i}}]\in K_{0}(U_{i})$ for all $i$; i.e., $[\alpha]$ is ``locally perfect'': the restrictions belong to the Grothendieck group of the derived category of perfect complexes on $X$. Does it follow that $[\alpha]\in K_{0}(X)$, i.e., $[\alpha]$ is ``perfect''?  
		
		We provide two types of counterexamples. First:
		\begin{proposition}
			Let $X=(\mathbf{P}^{1}\times\mathbf{P}^{1})\sqcup_{\mathbf{P}^{1}\times\{0\}}(\mathbf{P}^{1}\times\mathbf{P}^{1})$. Then there is an element in $K_{0}'(X)$ that is locally perfect but not perfect.
		\end{proposition}
		See Proposition \ref{glued} for a more detailed statement. With similar techniques we can construct an irreducible counterexample:		
		\begin{proposition}
			Let $\Gamma$ be a nodal curve with normalization $E\to\Gamma$ where $E$ an elliptic curve, and let $X=E\times\Gamma$. Then there is an element in $K_{0}'(X)$ that is locally perfect but not perfect.
		\end{proposition}
		See Proposition \ref{nodal} for a more detailed statement. And the second type of counterexample:
		\begin{proposition}
			Let $A$ be a Dedekind domain with distinct primes $\mathfrak{p},\mathfrak{q}\subset A$ such that $[\mathfrak{p}]=[\mathfrak{q}]\in\textnormal{Cl}(A)\setminus 2\textnormal{Cl}(A)$. Let $X=\textnormal{Spec}\left(A[\varepsilon]/(\varepsilon^2)\right)$. Then there exists an element in $K_{0}'(X)$ that is locally perfect but not perfect.
		\end{proposition}
		See Corollary \ref{thick} for a more detailed statement.
		
		\vspace{12pt}
		
		\noindent \textbf{Acknowledgements.}
		Thank you to Andrei Okounkov for suggesting this problem, to Johan de Jong for his invaluable guidance throughout the project, and to Elden Elmanto and Matthew Hase-Liu for helpful conversations.
		
	\section{Background}
		Most of the material in this section is stated and proved in \cite[\href{https://stacks.math.columbia.edu/tag/0FDE}{Tag 0FDE}]{stacks-project}.
		
		Let $X$ be a Noetherian scheme, and $D(X)=D^{b}(\textnormal{Coh}(X))$ be the bounded derived category of coherent sheaves on $X$. Let $D_{\textnormal{perf}}(X)$ be the full triangulated subcategory of perfect objects, i.e., those locally represented by bounded complexes of locally free coherent sheaves. 
		
		Let $K_{0}'(X):=K_{0}(D(X))$, i.e., the free abelian group on objects of $D(X)$ modulo the relations generated by $[B]-[A]-[C]$ for each distinguished triangle 
		$$A\to B\to C$$
		This group is sometimes also referred to as $G_{0}(X)$. Likewise, let $K_{0}(X):=K_{0}(D_{\textnormal{perf}}(X))$. For a ring $R$, we will let $K_{0}'(R):=K_{0}'(\textnormal{Spec}(R))$ and $K_{0}(R):=K_{0}(\textnormal{Spec}(R))$.
		
		Conveniently, $K_{0}'(X)=K_{0}(\textnormal{Coh}(X))$, i.e., the free abelian group generated by objects of $\textnormal{Coh}(X)$ modulo the relations generated by $[M]-[M']-[M'']$ for each exact short sequence 
		$$0\to M'\to M\to M''\to 0$$
		Likewise, for quasi-projective schemes,\footnote{More generally, for quasi-compact quasi-separated schemes with the resolution property.}  $K_{0}(X)=K_{0}(\textnormal{Vect}(X))$, where $\textnormal{Vect}(X)$ is the category of finite locally free sheaves.

		There is a natural inclusion map $K_{0}(X)\to K_{0}'(X)$. If $X$ is regular, then this map is surjective. Thus, since we are concerned with elements in the complement of the image of this map, we must consider singular schemes. 
		
		A \emph{finite} morphism $f:X\to Y$ induces a map $f_{*}:K_{0}'(X)\to K_{0}'(Y)$, since restriction of modules along a finite map preserves finiteness and $f_{*}$ is exact for $f$ finite. 
		
		A \emph{flat} morphism $f:X\to Y$ induces a map $f^{*}:K_{0}'(Y)\to K_{0}'(X)$, since base change along a flat map preserves finiteness and $f^{*}$ is exact. 
		
		\emph{Any} morphism $f:X\to Y$ induces a map $f^{*}:K_{0}(Y)\to K_{0}(X)$, since pullbacks of perfect complexes are perfect.
		
		Finally, the rank and determinant maps on $\textnormal{Vect}(X)$ descend to maps $\textnormal{rk}:K_{0}(X)\to\mathbf{Z}$ and $\det:K_{0}(X)\to \textnormal{Pic}(X)$ since they are additive on exact sequences. These maps are surjective and compatible with pullbacks.
		
		\begin{remark}
			The following is not necessary for our results, but provides an obstruction to constructing counterexamples. For a curve $X$, i.e., an integral separated scheme of dimension one of finite type over an algebraically closed field, there is a map from the divisor class group to the Grothendieck group
			$$\textnormal{Cl}(X)\to K_{0}'(X),\quad [x]\mapsto[\kappa(x)]$$ 
			If $x\in X$ is a singular point but $[x]=\sum_{i}[x_{i}]$ in $\textnormal{Cl}(X)$ and each $x_{i}$ is non-singular, then $[x]=\sum_{i}[x_{i}]\in\textnormal{im}(K_{0}(X))\subset K_{0}'(X)$. 
			
			Thus, we did not find any counterexamples on curves --- though it may be possible to do so on a singular curve with a divisor such that no linearly equivalent divisor is supported on the non-singular locus. 
		\end{remark}
		
	\section{Main Counterexamples}		
		Let $X$ be a Noetherian scheme with open non-singular locus $U$. Recall we have the following restriction/determinant/inclusion maps:
		$$
		\begin{tikzcd}
			K_{0}(X) \arrow[r, "\textnormal{res}"] \arrow[d, "\det"', two heads] & K_{0}(U) \arrow[d, "\det", two heads] & K_{0}(X) \arrow[r, "\textnormal{res}"] \arrow[d, "\textnormal{inc}"'] & K_{0}(U) \arrow[d, "\cong"', "\textnormal{inc}"] \\
			\textnormal{Pic}(X) \arrow[r, "\textnormal{res}"]                          & \textnormal{Pic}(U)                         & K_{0}'(X) \arrow[r, "\textnormal{res}"]                            & K_{0}'(U)                   
		\end{tikzcd}$$
		Note that $K_{0}(U)\xrightarrow{\cong}K_{0}'(U)$ since $U$ is regular.
		\begin{lemma}\label{image}
			Let $[\mathcal{M}]\in K_{0}'(X)$. Then 
			$$\det[\mathcal{M}|_{U}]\in\textnormal{Pic}(U)\setminus\textnormal{im}(\textnormal{Pic}(X))\implies[\mathcal{M}]\in K_{0}'(X)\setminus\textnormal{im}(K_{0}(X))$$
		\end{lemma}
		\begin{proof}
			 Suppose $[\mathcal{M}]$ is the image of some $[M]\in K_{0}(X)$. Then $[\mathcal{M}|_{U}]\in K_{0}(U)$ is the image of $[M]$. Thus $\det[\mathcal{M}|_{U}]$ is the image of $[M]$ under $K_{0}(X)\to K_{0}(U)\to\textnormal{Pic}(U)$, and is therefore the image of $\det[M]$ under $\textnormal{Pic}(X)\to\textnormal{Pic}(U)$, which is a contradiction.
		\end{proof}
		In the rest of this
		section we work with varieties and schemes over a fixed field $k$. Now, let $X:=Y\sqcup_{C}Z$, where $Y$ and $Z$ are smooth surfaces glued along a smooth curve $C\subset Y,Z$. Let $i_{Y},i_{Z}:C\to Y,Z$ be the closed immersions. Then $U=X\setminus C=(Y\setminus C)\sqcup(Z\setminus C)$.
		\begin{lemma}\label{Pic}
			In the above setup,
			$$\textnormal{Pic}(X)=\{(a,b)\in\textnormal{Pic}(Y)\times\textnormal{Pic}(Z):i_{Y}^{*}a=i_{Z}^{*}b\in\textnormal{Pic}(C)\}$$
			$$\textnormal{Pic}(U)=(\textnormal{Pic}(Y)/\langle[C]\rangle)\times(\textnormal{Pic}(Z)/\langle[C]\rangle) $$
			with the obvious restriction map.
		\end{lemma}
		\begin{proof}
			$\textnormal{Pic}(X)=\textnormal{Pic}(Y)\times_{\textnormal{Pic}(Y\cap Z)}\textnormal{Pic}(Z)$ is a general fact.
		\end{proof}
		The following lemma will be used to show local perfectness.
		\begin{lemma}\label{perfect}
			Let $D\subset Y$ be a prime divisor, and let $\mathcal{I}_{D}\subset\mathcal{O}_{X}$ be its ideal sheaf in $X$. If $\mathcal{O}_{Y}(D)\cong\mathcal{O}_{Y}$, then $[\mathcal{I}_{D}]\in\textnormal{im}(K_{0}(X))\subset K_{0}'(X)$.  
		\end{lemma}
		\begin{proof}
			Let $i:Y\to X$ be the closed immersion. Since this is finite, there is an induced map $i_{*}:K_{0}'(Y)\to K_{0}'(X)$. Under this map, 
			$$i_{*}([\mathcal{O}_{Y}/\mathcal{O}_{Y}(-D)])=[\mathcal{O}_{X}/\mathcal{I}_{D}]$$
			Under the hypothesis of the lemma the preimage is zero so $0=[\mathcal{O}_{X}/\mathcal{I}_{D}]=[\mathcal{O}_{X}]-[\mathcal{I}_{D}]\in K_{0}'(X)$ and $[\mathcal{I}_{D}]=[\mathcal{O}_{X}]\in\textnormal{im}(K_{0}(X))$.
		\end{proof}
		We now specialize to $Y\cong Z\cong C\times\mathbf{P}^{1}$ and $X:=Y\sqcup_{C\times\{0\}}Z$ (abuse of notation).
		\begin{lemma}\label{diagonal}
			With $X$ as above, 
			$$\textnormal{Pic}(U)=\textnormal{Pic}(C)\oplus\textnormal{Pic}(C)$$ 
			$$\textnormal{im}(\textnormal{Pic}(X)\to\textnormal{Pic}(\textnormal{U}))=\{(L,L):L\in\textnormal{Pic}(C)\} $$
		\end{lemma}
		\begin{proof}
			Note that $\textnormal{Pic}(C\times\mathbf{P}^{1})\cong\textnormal{Pic}(C)\oplus\mathbf{Z}$, the class $[C\times\{0\}]$ is a generator of $\mathbf{Z}$, and $i_{C\times\{0\}}^{*}$ is the projection $\textnormal{Pic}(C)\oplus\mathbf{Z}\xrightarrow{\pi}\textnormal{Pic}(C)$. Thus, by Lemma \ref{Pic},
			$$\textnormal{Pic}(X)=\textnormal{Pic}(C)\oplus\mathbf{Z}\oplus\mathbf{Z},\quad \textnormal{Pic}(U)=\textnormal{Pic}(C)\oplus\textnormal{Pic}(C) $$
			and restriction $\textnormal{Pic}(X)\to\textnormal{Pic}(U)$ is $\textnormal{Pic}(C)\oplus\mathbf{Z}\oplus\mathbf{Z}\xrightarrow{\pi}\textnormal{Pic}(C)\xrightarrow{\Delta}\textnormal{Pic}(C)\oplus\textnormal{Pic}(C) $.
		\end{proof}
	
		Now we can construct our counterexample.
		\begin{proposition}\label{glued}
			Let $p\in C$ such that $[p]\neq0\in\textnormal{Pic}(C)$. Let $D:=\{p\}\times\mathbf{P}^{1}\subset Y\subset X$, and let $\mathcal{M}:=\mathcal{I}_{D}\subset\mathcal{O}_{X}$ be its ideal sheaf in $X$. Then \begin{enumerate}
				\item $[\mathcal{M}]\in K_{0}'(X)\setminus\textnormal{im}(K_{0}(X))$, and \item $[\mathcal{M}|_{V_{i}}]\in\textnormal{im}(K_{0}(V_{i}))$ for $\bigcup_{i}V_{i}=X$ an open cover (defined below).
			\end{enumerate}						
		\end{proposition}
		\begin{proof}
			Note 
			$$\det[\mathcal{M}|_{U}]=\mathcal{O}_{U}(\{p\}\times(\mathbf{P}^{1}\setminus\{0\}))=([p],0)\in\textnormal{Pic}(C)\oplus\textnormal{Pic}(C)$$
			is not in the diagonal, so Lemma \ref{diagonal} and Lemma \ref{image} imply the first claim.
			
			Next, let $\bigcup_{i}C_{i}=C$ be an open cover such that $\mathcal{O}_{C}(p)|_{C_{i}}\cong\mathcal{O}_{C_{i}}$. Let $Y_{i}:= Z_{i}:=C_{i}\times\mathbf{P}^{1}$, set $V_{i}:=Y_{i}\sqcup_{C_{i}\times\{0\}}Z_{i}$, and let $D_{i}:=D\cap Y_{i}$. We have
			$$\textnormal{Pic}(Y_{i})=\textnormal{Pic}(C_{i})\oplus\mathbf{Z}$$
			and $[D_{i}]=([p],0)=0$ by hypothesis. Thus, $\mathcal{O}_{Y_{i}}(D_{i})\cong\mathcal{O}_{Y_{i}}$, so Lemma \ref{perfect} implies the second claim. 
		\end{proof}
		\begin{example}
			Let $C$ be $\mathbf{P}^{1}$. Then any $p\in C$ satisfies the hypotheses of Proposition \ref{glued}.
		\end{example}
		We can use the same idea to construct a counterexample on an irreducible $X$. First we prove a lemma.
		\begin{lemma}\label{locally-trivial}
			Let $\nu:\tilde{W}\to W$ be a finite map of schemes, and let $\mathcal{L}\in\textnormal{Pic}(W)$. Then there is an open cover $\bigcup_{j}W_{j}=W$ such that $\nu_{*}[\nu^{*}\mathcal{O}_{W}]|_{W_{i}}=\nu_{*}[\nu^{*}\mathcal{L}]|_{W_{i}}$ in $K_{0}'(W_{i})$.
		\end{lemma}
		\begin{proof}
			Let $\{W_{i}\}_{i}$ be an open affine cover such that $\mathcal{L}|_{W_{i}}\cong\mathcal{O}_{W_{i}}$. Let $W_{i}=\text{Spec}(R_{i})$, and let $\nu^{-1}(W_{i})=\text{Spec}(\tilde{R}_{i})$. Then $H^{0}(\nu_{*}\nu^{*}\mathcal{L}|_{W_{i}}) = H^{0}(\nu^{*}\mathcal{L}|_{\text{Spec}(\tilde{R}_{i})})=\tilde{R}_{i}$ and likewise $H^{0}(\nu_{*}\nu^{*}\mathcal{O}_{W}|_{W_{i}})=\tilde{R}_{i}$. Since a coherent sheaf on an affine scheme is determined by its global sections, we have the desired equality.
		\end{proof}
		\begin{proposition}\label{nodal}
			Let 
			\begin{itemize}
				\item $E$ be an elliptic curve $p_{1},p_{2}\in E$ with $p_{1}\neq p_{2}$ but $p_{1}\in\langle p_{2}\rangle$ and $p_{2}\in\langle p_{1}\rangle$ under the group law of $E$, 
				\item $\Gamma$ be a curve with a single simple node at $\gamma$ and normalization $\nu':E\to\Gamma$ with $\nu'^{-1}(\gamma)=\{p_{1},p_{2}\}$, and
				\item $X:=E\times\Gamma$ with smooth locus $U$ and normalization $\nu:\tilde{X}=E\times E\to E\times\Gamma$.
			\end{itemize}
			Let $\Delta\subset E\times E$ be the diagonal, $D:=\nu(\Delta)\subset X$, and $\mathcal{M}:=\mathcal{I}_{D}\subset\mathcal{O}_{X}$ be its ideal sheaf in $X$. Then    
			\begin{enumerate}
				\item $[\mathcal{M}]\in K_{0}'(X)\setminus\textnormal{im}(K_{0}(X))$, and \item $[\mathcal{M}|_{V_{i,j}}]\in\textnormal{im}(K_{0}(V_{i,j}))$ for $\bigcup_{i,j}V_{i,j}=X$ an open cover (defined below).
			\end{enumerate}		
		\end{proposition}
		\begin{proof}
			For the first claim, by Lemma \ref{image} it suffices to show $\det[\mathcal{M}|_{U}]\in\text{Pic}(U)\setminus\textnormal{im}(\text{Pic}(X))$. The restriction factors as 
			$$\text{Pic}(X)\xrightarrow{\nu^{*}}\text{Pic}(\tilde{X})\to\text{Pic}(U) $$
			We have 
			$$\text{im}(\nu^{*}) =
			\{\mathcal{L}\in\text{Pic}(\tilde{X}):\iota_{E\times\{p_{1}\}}^{*}\mathcal{L}\cong \iota_{E\times\{p_{2}\}}^{*}\mathcal{L}\}$$
			and $\text{Pic}(\tilde{X})\to\text{Pic}(U)$ is the quotient by $\langle \mathcal{O}_{\tilde{X}}(E\times\{p_{1}\}),\mathcal{O}_{\tilde{X}}(E\times\{p_{2}\})\rangle$. Note that $\iota_{E\times\{p_{i}\}}^{*}(\mathcal{O}_{\tilde{X}}(E\times\{p_{j}\}))=0$ for $i,j\in\{1,2\}$, so if $\overline{\mathcal{L}}\in\text{Pic}(U)$ is in $\text{im}(\text{Pic}(X))$, then any lift $\mathcal{L}\in\text{Pic}(\tilde{X})$ of $\overline{\mathcal{L}}$ is in $\text{im}(\nu^{*})$. Thus, it suffices to show $\mathcal{O}_{\tilde{X}}(-\Delta)\not\in\text{im}(\nu^{*})$, and which is true since 
			$$\iota_{E\times\{p_{1}\}}^{*}\mathcal{O}_{\tilde{X}}(-\Delta)\cong\mathcal{O}_{E}(-p_{1})\not\cong\mathcal{O}_{E}(-p_{2})\cong\iota_{E\times\{p_{2}\}}^{*}\mathcal{O}_{\tilde{X}}(-\Delta) $$
			
			For the second claim, for $i=1,2$ let $V_{i}:=(E\setminus\{p_{i}\})\times \Gamma\subset X$, let $\nu_{i}:\tilde{V}_{i}\to V_{i}$ be the normalization, and let $\Delta_{i}:=\Delta\cap\tilde{V}_{i}$ and $D_{i}:=D\cap V_{i}$. 
			
			Since
			$$\mathcal{O}_{\tilde{V}_{i}}(-\Delta_{i})|_{\tilde{V}_{i}\cap(E\times\{p_{1}\})}\cong\mathcal{O}_{E\setminus\{p_{i}\}}\cong\mathcal{O}_{\tilde{V}_{i}}(-\Delta_{i})|_{\tilde{V}_{i}\cap(E\times\{p_{2}\})}$$
			we have $\mathcal{O}_{\tilde{V}_{i}}(-\Delta_{i})=\nu_{i}^{*}\mathcal{L}$ for some $\mathcal{L}\in\text{Pic}(V_{i})$. Thus, from the exact sequence for the divisor $\Delta_{i}$ we have
			$$0\to{\nu_{i}}_{*}\nu_{i}^{*}\mathcal{L}\to{\nu_{i}}_{*}\mathcal{O}_{\tilde{V}_{i}}\to{\nu_{i}}_{*}\mathcal{O}_{\Delta_{i}}\to0  $$
			noting that ${\nu_{i}}_{*}$ is exact by finiteness of $\nu_{i}$. Since $\nu_{i}$ induces an isomorphism $\Delta_{i}\to D_{i}$, we have ${\nu_{i}}_{*}\mathcal{O}_{\Delta_{i}}\cong\mathcal{O}_{D_{i}}$. And since $\mathcal{O}_{\tilde{V}_{i}}\cong\nu_{i}^{*}\mathcal{O}_{V_{i}}$, we have $[\mathcal{O}_{D_{i}}]={\nu_{i}}_{*}[\nu_{i}^{*}\mathcal{O}_{V_{i}}]-{\nu_{i}}_{*}[\nu_{i}^{*}\mathcal{L}]\in K_{0}'(V_{i})$. Thus, by Lemma \ref{locally-trivial} there is an open cover $V_{i}=\bigcup_{i,j}V_{i,j}$ such that $[\mathcal{O}_{D_{i}}|_{V_{i,j}}]=0\in K_{0}'(V_{i,j})$.
 			
			Finally, we have 
			\begin{align*}
				[\mathcal{M}|_{V_{i,j}}] &=
				[\mathcal{O}_{V_{i}}(-D_{i})|_{V_{i,j}}] \\&=
				[\mathcal{O}_{V_{i}}|_{V_{i,j}}]-[\mathcal{O}_{D_{i}}|_{V_{i,j}}] \\&=
				[\mathcal{O}_{V_{i,j}}]
			\end{align*}
			So $[\mathcal{M}|_{V_{i,j}}]\in\text{im}(K_{0}(V_{i,j}))$.
		\end{proof}
		\begin{remark}
			Many concrete examples of such $X$ exist: letting $E$ be any elliptic curve, choose any $p\in E$ not of order $2$, and let $p_{1}:=p$ and $p_{2}:=-p$; then it is possible to construct a birational morphism $f:E\to\mathbf{P}^{N}$ such that $f(p_{1})=f(p_{2})$, and we can take $\Gamma$ to be the image of $f$ (after normalizing additional singularities). 
		\end{remark}
	\section{Additional Counterexample}
		In this section we provide another type of counterexample using (non-reduced) thickenings of Dedekind domains.
		
		Let $A$ be a regular ring, and let $B:=A[z]/(z^2)$ be the dual numbers over $A$, which we will write as $B=A[\varepsilon]$ where $\varepsilon$ is the image of $z$.
		\begin{lemma}\label{restriction}
			Restriction along $A\to B$ and $B\to A$ induce isomorphisms $K_{0}'(B)\xrightarrow{\sim}K_{0}'(A)$ and $K_{0}'(A)\xrightarrow{\sim}K_{0}'(B)$.
		\end{lemma}
		\begin{proof}
			The composition $A\to B\to A$ is the identity, so by functoriality, the induced composition $K_{0}'(A)\leftarrow K_{0}'(B)\leftarrow K_{0}'(A)$ is too. But $K_{0}'(A)\to K_{0}'(B)$ is surjective since for any $[M]\in K_{0}'(B)$ we have $[M]=[\varepsilon M]+[M/(\varepsilon)]$ and both the latter are in the image. Since the map must be injective for the composition to be the identity, it is an isomorphism, and so the other map is too.
		\end{proof}
		\begin{lemma}
			The map $A\to B$ is flat.
		\end{lemma}
		\begin{proof}
			$B$ is a free $A$-module.
		\end{proof}
		Thus, base change along $A\to B$ gives a well-defined map $K_{0}'(A)\to K_{0}'(B)$.
	
		\begin{lemma}\label{double}
			The map $K_{0}'(A)\to K_{0}'(B)$ given by $[M]\mapsto [M\otimes_{A}B]$ maps $\alpha\mapsto 2\alpha$ under the isomorphism $K_{0}'(B)\cong K_{0}'(A)$ of Proposition \ref{restriction}.
		\end{lemma}
		\begin{proof}
			For any $B$-module $M$ we have $M\otimes_{A} B\cong M\oplus M$ as $A$-modules, so $[M]\mapsto 2[M]$.  
		\end{proof}
		Recall that the base change map $K_{0}(A)\to K_{0}(B)$ is well-defined regardless of flatness.
		\begin{lemma}\label{surjective}
			The map $K_{0}(A)\to K_{0}(B)$ given by $[M]\mapsto [M\otimes_{A}B]$ is an isomorphism.
		\end{lemma}
		\begin{proof}
			Consider the map $K_{0}(B)\to K_{0}(A)$ given by $[N]\mapsto[N\otimes_{B}A]$. The composition $K_{0}(A)\to K_{0}(B)\to K_{0}(A)$ is the identity, so it suffices to show that the latter map is injective, i.e., for projective $B$-modules $N,N'$, that $N\otimes_{B}A\cong N'\otimes_{B}A$ as $A$-modules implies $N\cong N'$ as $B$-modules.
			
			If $N\otimes_{B}A\cong N'\otimes_{B}A$ as $A$-modules, then $N/(\varepsilon)\cong N'/(\varepsilon)$ as $B$-modules. Since $N$ is projective, we can lift the composition $N\to N'/(\varepsilon)$ to a map $\varphi:N\to N'$:
			$$\begin{tikzcd}
				0 \arrow[r] & \varepsilon N \arrow[r] \arrow[d, dashed] & N \arrow[r] \arrow[d, "\exists\varphi", dashed] \arrow[rd] & N/(\varepsilon) \arrow[r] \arrow[d, "\cong"] & 0 \\
				0 \arrow[r] & \varepsilon N' \arrow[r]                   & N' \arrow[r]                                               & N'/(\varepsilon) \arrow[r]                   & 0
			\end{tikzcd}$$
			By the five lemma, it suffices to check that $\varphi$ induces an isomorphism $\varepsilon N\to\varepsilon N'$. We can check this locally, and since $N,N'$ are locally free we can assume they are free. Since $\varepsilon B\cong B/(\varepsilon)$, so $\varepsilon B^{n}\cong B^{n}/(\varepsilon)$, and $\varphi:\varepsilon N\to\varepsilon N'$ corresponds to $N/(\varepsilon)\xrightarrow{\cong}N'/(\varepsilon)$.					
		\end{proof}
		\begin{proposition}\label{two}
			The image of the inclusion $K_{0}(B)\to K_{0}'(B)$ is $2K_{0}'(B)$.
		\end{proposition}
		\begin{proof}
			The base change and inclusion maps yield a commutative diagram.
			$$
			\begin{tikzcd}
				K_{0}'(A) \arrow[r]                   & K_{0}'(B)          \\
				K_{0}(A) \arrow[u, "\cong"] \arrow[r, "\cong"] 	& K_{0}(B) \arrow[u]
			\end{tikzcd} $$
			with $K_{0}(A)\xrightarrow{\cong} K_{0}'(A)$ following from regularity of $A$ and $K_{0}(A)\xrightarrow{\cong} K_{0}(B)$ following from Lemma \ref{surjective}. Thus, the image of $K_{0}(B)\to K_{0}'(B)$ equals the image of the composition $K_{0}(A)\to K_{0}'(B)$, which equals the image of $K_{0}'(A)\to K_{0}'(B)$. And by Lemma \ref{double}, this map is multiplication by $2$.			
		\end{proof}
		Now, assume $A$ is a Dedekind domain. It is well known that $K_{0}(A)\cong \mathbf{Z}\oplus \textnormal{Cl}(A)$, and that an open subscheme of $\text{Spec}(A)$ is the spectrum of a Dedekind domain. 
		\begin{corollary}\label{thick}
			Let $A$ be a Dedekind domain with distinct primes $\mathfrak{p},\mathfrak{q}\subset A$ such that
			\begin{itemize}
				\item $[\mathfrak{p}]=[\mathfrak{q}]\in\textnormal{Cl}(A)$, and 
				\item $[\mathfrak{  p}],[\mathfrak{q}]\notin2\textnormal{Cl}(A)$.
			\end{itemize}
			Let $B=A[\varepsilon]$ with $\mathfrak{p}':=(\varepsilon,\mathfrak{p})\subset B$ and $\mathfrak{q}':=(\varepsilon,\mathfrak{q})\subset B$, and let $U_{1}:=\textnormal{Spec}(B)\setminus\{\mathfrak{p}'\}$ and $U_{2}:=\textnormal{Spec}(B)\setminus\{\mathfrak{q}'\}$. Then $U_{1}\cup U_{2}=\textnormal{Spec}(B)$ and
			\begin{enumerate}
				\item $[\mathfrak{p}']\in K_{0}'(B)\setminus\textnormal{im}(K_{0}(B))$, and 
				\item $[\mathfrak{p}'|_{U_{i}}]\in\textnormal{im}(K_{0}(U_{i}))$ for $i=1,2$.
			\end{enumerate} 
		\end{corollary}
		\begin{proof}
			By Lemma \ref{restriction}, restriction along $A\to B$ induces $K_{0}'(B)\cong K_{0}(A)\cong\mathbf{Z}\oplus\textnormal{Cl}(A)$, and by Proposition \ref{two}, $\textnormal{im}(K_{0}(B))=2\mathbf{Z}\oplus2\textnormal{Cl}(A)$. We have a short exact sequence of $B$-modules:
			$$0\to\varepsilon B\to \mathfrak{p}'\to \mathfrak{p}'/(\varepsilon)\to0  $$
			Then $[\varepsilon B],[\mathfrak{p}'/(\varepsilon)]\in K_{0}'(B)$ restrict to $[A],[\mathfrak{p}]\in K_{0}(A)$, and since $[A]+[\mathfrak{p}]\notin2\mathbf{Z}\oplus 2\textnormal{Cl}(A)$ we have $[\mathfrak{p}']=[\varepsilon B]+[\mathfrak{p}'/(\varepsilon)]\notin\textnormal{im}(K_{0}(B))$.
			
			And $[\mathfrak{p}'|_{U_{1}}]=[\mathcal{O}_{U_{1}}]\in\textnormal{im}(K_{0}(U_{1}))$ and $[\mathfrak{p}'|_{U_{2}}]=[\mathfrak{q}'|_{U_{2}}]=[\mathcal{O}_{U_{2}}]\in\textnormal{im}(K_{0}(U_{2}))$.
		\end{proof}
		Here is a concrete case: 
		\begin{example}
			Let $\tilde{A}=\mathbf{Z}[\sqrt{-21}]$. It is known that $\textnormal{Cl}(\tilde{A})=(\mathbf{Z}/2\mathbf{Z})^{2}$ and that each element in the class group is represented by a prime ideal. Let $\tilde{\mathfrak{p}}$, $\tilde{\mathfrak{q}}$, and $\tilde{\mathfrak{r}}$ be prime ideals representing the three non-trivial classes, and note that none are in $2\textnormal{Cl}(\tilde{A})$. Let $\textnormal{Spec}(A)=\textnormal{Spec}(\tilde{A})\setminus\{\tilde{\mathfrak{r}}\}$, and let $\mathfrak{p},\mathfrak{q}\in\textnormal{Spec}(A)$ be the images of the other primes. Then $[\mathfrak{p}]=[\mathfrak{q}]\in\textnormal{Cl}(A)$ and $[\mathfrak{p}],[\mathfrak{q}]\notin2\textnormal{Cl}(A)$, satisfying the hypotheses of Corollary \ref{thick}.
		\end{example}
	
	\nocite{*}
	\bibliography{bibliography}{}
	\bibliographystyle{plain}
		
\end{document}